\documentclass[11pt,oneside,english,dvipsnames,a4paper,reqno]{amsart}
\usepackage[T1]{fontenc}
\usepackage[latin9]{inputenc}
\usepackage{geometry}
\usepackage{refstyle}
\usepackage{units}
\usepackage{amstext}
\usepackage{amsthm}
\usepackage{amssymb}
\usepackage{setspace}
\usepackage{hyperref}
\usepackage{enumerate}
\usepackage{dsfont}
\usepackage[dvipsnames]{xcolor}

\usepackage{tikz}
\usetikzlibrary{decorations}
\usetikzlibrary{decorations.shapes}
\usetikzlibrary{arrows.meta}

\makeatletter

\numberwithin{equation}{section}
\numberwithin{figure}{section}
\theoremstyle{plain}
\newtheorem{thm}{\protect\theoremname}[section]
  \theoremstyle{remark}
  \newtheorem{rem}[thm]{\protect\remarkname}
  \theoremstyle{definition}
  \newtheorem{defn}[thm]{\protect\definitionname}
  \theoremstyle{plain}
  \newtheorem{lem}[thm]{\protect\lemmaname}
  \theoremstyle{plain}
  \newtheorem{prop}[thm]{\protect\propositionname}
  \theoremstyle{remark}
  \newtheorem{observation}[thm]{\protect\observationname}
  \theoremstyle{remark}
  \newtheorem{claim}[thm]{\protect\claimname}
  \theoremstyle{remark}
  \newtheorem{ex}[thm]{\protect\examplename}

 \providecommand{\claimname}{Claim}
  \providecommand{\definitionname}{Definition}
  \providecommand{\lemmaname}{Lemma}
  \providecommand{\observationname}{Observation}
  \providecommand{\propositionname}{Proposition}
  \providecommand{\remarkname}{Remark}
  \providecommand{\examplename}{Example}
\providecommand{\theoremname}{Theorem}

\title{Self-diffusion Coefficient in the Kob-Andersen Model}
\author{Anatole Ertul}
\email{ertul@math.univ-lyon1.fr}
\author{Assaf Shapira}
\email{assafshap@gmail.com}

\global\long\def\One{\mathds{1}}

\global\long\def\zz{\mathbb{Z}}

\global\long\def\rr{\mathbb{R}}

\global\long\def\nn{\mathbb{N}}

\global\long\def\l{\ell}

\global\long\def\BP{\operatorname{BP}}

\def\restriction#1#2{\mathchoice
              {\setbox1\hbox{${\displaystyle #1}_{\scriptstyle #2}$}
              \restrictionaux{#1}{#2}}
              {\setbox1\hbox{${\textstyle #1}_{\scriptstyle #2}$}
              \restrictionaux{#1}{#2}}
              {\setbox1\hbox{${\scriptstyle #1}_{\scriptscriptstyle #2}$}
              \restrictionaux{#1}{#2}}
              {\setbox1\hbox{${\scriptscriptstyle #1}_{\scriptscriptstyle #2}$}
              \restrictionaux{#1}{#2}}}
\def\restrictionaux#1#2{{#1\,\smash{\vrule height .8\ht1 depth .85\dp1}}_{\,#2}}

\begin{document}

\begin{abstract}
The Kob-Andersen model is a fundamental example of a kinetically constrained lattice gas, that is, an interacting particle system with Kawasaki type dynamics and kinetic constraints. In this model, a particle is allowed to jump when sufficiently many neighboring sites are empty. We study the motion of a single tagged particle and in particular its convergence to a Brownian motion. Previous results showed that the path of this particle indeed converges in diffusive time-scale, and the purpose of this paper is to study the rate of decay of the self-diffusion coefficient for large densities. We find upper and lower bounds matching to leading behavior.
\end{abstract}

\maketitle

\section{Introduction}

Kinetically constrained lattice gases is a family of models divised by physicists in order to study glassy systems, that could be seen as the conservative version of kinetically constrained spin models, see e.g. \cite{KA, RS03, GST}. In this paper we study one such model -- the $(k,d)$-Kob-Andersen model. It is a Markov process living on the graph $\zz^d$, that depends on a parameter $k \ge 2$. Each site of $\zz^d$ may contain at most one particle, that can jump to an empty neighboring site if it has at least $k$ empty neighbors both before and after the jump. When this constraint is satisfied, the particle jumps at rate $1$. For any $q\in(0,1)$, this process is reversible with respect to the product Bernoulli measure of parameter $1-q$.\\

When $q$ is small, the constraint is difficult to satisfy, resulting in a significant lengthening of time scales related to this process. Indeed, for a particle to move around it must wait for sufficiently many vacancies to arrive at its vicinity. This fact gives rise to two important length scales of this model -- for these vacancies to propagate and reach the particle they must form a droplet of some typical length scale $\ell$. In the case $k=d=2$, for example, for a droplet to advance it must find a close by vacancy, which is typically possible for $\ell \approx 1/q$. 
The second scale is the distance $L$ within which such a droplet can be found, i.e., $L\approx q^{-\ell}$, which in the case $k=d=2$ is (to leading behavior) $e^{-1/q}$.
For higher values of $k$ and $d$ the mechanism which allows a droplet to move is based on the fact that a $d-1$ dimensional layer parallel to the droplet could evolve like a $(k-1,d-1)$-Kob-Andersen model since one of the $k$ required empty neighbors comes from the droplet. Particles are thus allowed to move in this layer if its size reaches the scale $L$ of the $(k-1,d-1)$ dynamics, which thus equals $\ell$ of the $(k,d)$ dynamics. The details of this argument can be found in \cite{BFT} (see also \cite{SphD}).\\ 

At large times, the path of a marked particle converges to a Brownian motion with a coefficient called the self-diffusion, which is the subject of this paper. In \cite{BT18} it has been proven that this coefficient is strictly positive for all $q\in(0,1)$, in contrast to the conjecture in the physics literature that below some non-zero critical $q$ the path of tagged particles is no longer diffusive.
In this work we find the dependence of this diffusion coefficient in $q$, showing that it decays very fast when $q$ is small, in a similar way to the spectral gap \cite{MST}.\\

We start by introducing the model and our result, and then prove a lower and an upper bound on the diffusion coefficient. The main tool we use is a variation formula of \cite{S90} for the diffusion coefficient.
In order to bound it from below, as in \cite{BT18}, we compare the Kob-Andersen dynamics with a random walk on an infinite percolation cluster. The upper bound is obtained by identifying an appropriate test function related to the bootstrap percolation, a process which is closely related to the Kob-Andersen model.

\section{Model and main result}

The model we study here is defined on the lattice $\mathbb{Z}^d$. We denote by $(e_1,\dots,e_d)$ the standard orthonormal basis. The set of configurations is $\Omega = \{0,1\}^{\mathbb{Z}^d}$, where $0$ stands for an empty site and $1$ for an occupied site. Given two sites $x,y$ we denote $x \sim y$ if they are nearest neighbors. We also denote by $[L]^d$ the cube $[1,L]^d$.\\

Fix and integer $k \in [2,d]$. For $\eta \in \Omega$ and $x\sim y$, we define the local constraint for the edge $xy$ by

\begin{equation}\label{def:constraint}
c_{xy}(\eta)=\begin{cases}
1 & \text{if } \sum_{z \sim x, z \ne y} (1 - \eta(z)) \ge k-1 \,\, \text{and} \,\, \sum_{z \sim y, z\ne x} (1 - \eta(z)) \ge k - 1,\\
0 & \text{otherwise}.
\end{cases}
\end{equation}
For $V\subset \zz^d$, and $\eta \in \{0,1\}^V$, we define the constraint $c_{xy}(\eta) = c_{xy}(\eta')$ where $\eta' \in \Omega$ equal $\eta$ on $V$ and is entirely occupied elsewhere.

The generator of our Markov process describing the KA dynamics operating on a local function $f$ is given by:
\begin{equation}\label{generator}
\mathcal{L}f(\eta) = \sum_{x \in \mathbb{Z}^d} \sum_{y \sim x} c_{xy}(\eta)\eta(x)(1-\eta(y))\left[ f(\eta^{xy}) - f(\eta) \right],
\end{equation}
where $\eta^{xy}$ is the configuration equal to $\eta$ except that $\eta^{xy}(x) = \eta(y)$ and $\eta^{xy}(y) = \eta(x)$. In words, the only way a configuration can change is a particle (i.e. an occupied site) "jumping" to a neighboring empty site provided each of those site have at least $(k-1)$ other empty neighbors. We call this transition a \emph{legal KA-$k$f transition}, or simply \emph{legal transition} whenever the context allows it.\\

Observe that, from Formula \ref{generator}, this process is reversible with respect to the product measure $\mu := \otimes_{x \in \mathbb{Z}^d} \text{Ber}(1-q)$ for any $q \in (0,1)$.\\

We now consider the trajectory of a tagged particle. Let $\mu_0 = \mu(\cdot|\eta(0)=1)$ and, under the initial distribution $\mu_0$, $X_t$ the position at time $t$ of the particle initially at $0$. More precisely, $(X_t,\eta_t)_{t \ge 0}$ is the Markov process with generator:

\begin{multline}
    \mathcal{L}_{\text{tagged}}f(X,\eta) = \sum_{\substack{y \in \zz^d \\ y\ne X}}\sum_{z \sim y} c_{yz}(\eta)\eta(y)(1-\eta(z))\left[ f(X,\eta^{yz}) - f(X,\eta) \right] \\
    + \sum_{y \sim X} c_{Xy}(\eta)\eta(X)(1-\eta(y))\left[f(y,\eta^{Xy}) - f(X,\eta)\right]
\end{multline}

The following classic result gives a convergence for $X_t$:
\begin{thm}\cite{S90,KV}\label{cvg}
For any $q \in (0,1)$, there exists a non-negative $d\times d$ matrix $D(q)$ such that 
$$\varepsilon X_{\varepsilon^{-2}t} \underset{\varepsilon \rightarrow 0}{\longrightarrow} \sqrt{2D(q)}B_t,$$
where $B_t$ is a $d$-dimensional Brownian motion process and the convergence holds in the sense of weak convergence of path measures on $D(\mathbb{R}_+,\mathbb{R}^d)$. Furthermore, $D(q)$ is characterized by the following variational formula:
\begin{equation}
\forall u \in \mathbb{R}^d, \ u\cdot Du=\inf_{f}\mu_{0}\left[\sum_{x\neq0}\sum_{y\sim x}c_{xy}\left(f\left(\eta^{xy}\right)-f\left(\eta\right)\right)^{2}+\sum_{y\sim0}c_{0y}\left(u\cdot y+f\left(\tau_{y}\eta^{0y}\right)-f\left(\eta\right)\right)^{2}\right]\label{eq:variational_principle}, 
\end{equation}
where the infimum is taken over all local functions on $\Omega$, and $(\tau_y \eta)$ is the configuration defined by $(\tau_y \eta)(z) = \eta(z-y)$ for all $z\in \zz^d$.

\end{thm}

\begin{rem}
A priori, the diffusion coefficient is a matrix. In our case, however,
the model is invariant under permutation and inversion of the standard
basis vectors. This forces the diffusion matrix to be scalar, equal
to any arbitrary diagonal element.
\end{rem}

In \cite{BT18}, it was first proved that $D(q) > 0$ for all $q > 0$. We will give in this paper the appropriate scale of $D(q)$ when $q \rightarrow 0$. The main result is the following:

\begin{thm}\label{mainthm}
Let $d \ge 2$, $k\in [2,d]$.
For $q \in (0,1)$ let $D(q)$ the diffusion coefficient given by Theorem \ref{cvg}. Then for $q$ sufficiently small:\\

if $k = 2$: 
\[
1 / \exp\left(c\left(\log\nicefrac{1}{q}\right)^d q^{-\frac{1}{d-1}}\right)\le D(q)\le 1 / \exp\left(c' q^{-\frac{1}{d-1}} \right),
\]

if $k\ge3$: 

\[
1 / \exp_{(k-1)}\left(cq^{-\frac{1}{d-k+1}}\right)\le D(q)\le 1 / \exp_{(k-1)}\left(c'q^{-\frac{1}{d-k+1}}\right),
\]
where $\exp_{(k-1)}$ denotes the exponential function iterated $(k-1)$ times, and $c,c'$ are constants only dependent on $k$ and $d$.
\end{thm}

\section{Proof of the lower bound}

The proof of the lower bound will closely follow the proof of \cite{BT18}, sections 4 and 5. However, we use more refined combinatorial properties of the KA model in order to obtain the correct scaling.\\

Throughout the proof, $c$ and $\lambda$ denote generic positive constants which only depends on $d$ and $k$.

We start by defining a coarse grained version of the lattice, depending on two scales:


\begin{align}\label{def:l}
 \l &= \begin{cases} c\log\left(\nicefrac{1}{q}\right)\,q^{-\frac{1}{d-1}} & \text{if} \quad k=2, \\
 c\exp_{(k-2)}\left(q^{-\frac{1}{d-k+1}}\right) & \text{if} \quad k\ge3,\end{cases}\\
 \label{def:L} L &=   q^{-c \l}  .
\end{align}

\begin{defn}
A \textit{block} is a set of the form $\left(L+1\right)i+\left[L\right]^{d}$,
$i\in\zz^{d}$. A block is divided in \textit{boxes}, which are sets
of the form $\left(L+1\right)i+\l a+\left[\l\right]^{d}$, $i\in\zz^{d},a\in [L/\l]$.\\
We also call \emph{external face} of a box any connected component of the set of vertices at (graph) distance $1$ from the box.
See Figure \ref{fig:blocks_and_boxes}.
\end{defn}

\begin{figure}
\begin{tikzpicture}[scale=0.3, every node/.style={scale=0.6}]	
	\tikzset{small|/.style={decoration={markings,mark=at position 1 with %
    {\arrow[scale=0.6]{|}}},postaction={decorate}}}
    
	\draw[step=1, lightgray, very thin] (-0.5,-0.5) grid +(18,17);
	\draw[step=15, black, very thick,shift={(1,1)}] (0,0) grid +(15,15);
	\draw[step=15, black, very thick,shift={(0,1)}] (0,0) grid +(-0.5,15);
	\draw[step=15, black, very thick,shift={(1,0)}] (0,0) grid +(15,-0.5);
	\draw[step=15, black, very thick] (0,0) grid +(-0.5,-0.5);
	\draw[step=15, black, very thick, shift={(17,1)}] (0,0) grid +(0.5,15);
	\draw[step=15, black, very thick, shift={(17,0)}] (0,0) grid +(0.5,-0.5);
	\draw[step=3, black, shift={(1,1)}] (0,0) grid +(15,15);
	\draw[step=3, black, shift={(0,1)}] (0,0) grid +(-0.5,15);
	\draw[step=3, black, shift={(1,0)}] (0,0) grid +(15,-0.5);
	\draw[step=3, black, shift={(17,1)}] (0,0) grid +(0.5,15);

	\draw[|-] (18.2,1) -- (18.2,8);
	\draw (18.2,8.5) node{$L$};
	\draw[-|] (18.2,9) -- (18.2,16);
	
	\draw[-|] (5,16.7) -- (4,16.7);
	\draw (5.5,16.7) node{$\ell$};
	\draw[-|] (6,16.7) -- (7,16.7);
	
	\draw (0.5,0.5) node{$i$};
	\draw (16.5,0.5) node{$j$};
	
	\fill [color=OliveGreen, opacity=0.2] (0,1) rectangle +(1,3);
	\fill [color=OliveGreen, opacity=0.2] (1,0) rectangle +(3,1);
	\fill [color=OliveGreen, opacity=0.2] (13,0) rectangle +(3,1);
	\fill [color=OliveGreen, opacity=0.2] (16,1) rectangle +(1,3);
	
	\fill [color=RoyalBlue, opacity=0.2] (1,1) rectangle +(3,3);
	\fill [color=RoyalBlue, opacity=0.2] (1,4) rectangle +(3,3);
	\fill [color=RoyalBlue, opacity=0.2] (1,7) rectangle +(3,3);
	\fill [color=BrickRed, opacity=0.2] (4,7) rectangle +(3,3);
	\fill [color=RoyalBlue, opacity=0.2] (7,7) rectangle +(3,3);
	\fill [color=RoyalBlue, opacity=0.2] (10,7) rectangle +(3,3);
	\fill [color=RoyalBlue, opacity=0.2] (10,4) rectangle +(3,3);
	\fill [color=RoyalBlue, opacity=0.2] (13,4) rectangle +(3,3);
	\fill [color=RoyalBlue, opacity=0.2] (13,1) rectangle +(3,3);
	
\end{tikzpicture}

\caption{\label{fig:blocks_and_boxes}Block-connected vertices $i$ and $j$. The boxes in blue are $(d,k)$-good, and the external faces in green are $(d-1,k-1)$-good. The red box is $(d,k)$-frameable.}

\end{figure}
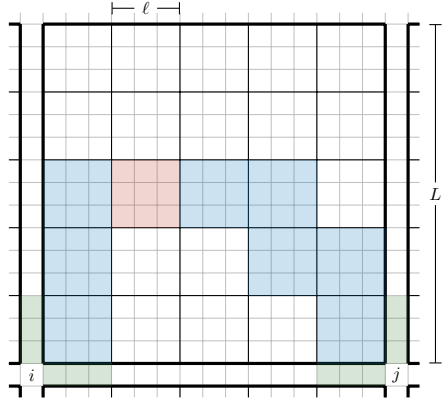

We think of the blocks' corners $(L+1)i$ as vertices of a graph $\zz_\l^d=(L+1)\zz^{d}$. More precisely, the graph $\zz_\l^d$ is defined as follows :
\begin{itemize}
    \item Its vertices are sites of the form $(L+1)i$, $i\in\zz^d$.
    \item Its edges connect a vertex $(L+1)i$ with a vertex $(L+1)j$ for $||i-j||_1 = 1$.
\end{itemize}

From now on, sites on the original lattice will be denoted using the letters $x,y, \dots$, while vertices of $\zz_\l^d$ will be denoted with the letters $i,j,\dots$. Similarly to \cite{BT18}, the purpose of this coarse grained lattice is to define an auxiliary dynamics that has diffusive behavior on a larger scale.\\

Let us now recall a few definitions in relation with the coarse grained lattice, introduced in \cite{MST}.

 \begin{defn}
 Let $E$ be a subset of the standard basis with size $|E|\le d-1$, $V\subset \zz^d$ a set of sites, and fix a site $x\in V$. The $|E|$-dimensional \emph{slice} of $V$ passing through $x$ in the directions of $E$ is defined as $V\cap (x+\text{span}E)$, where $\text{span}E$ is the linear span of $E$.
 \end{defn}

\begin{defn}
Given the $d$-dimensional cube $\mathcal{C}_n=[n]^d$ and an integer $k\le d$ we define the $k^{\textup{th}}$ \emph{frame} of $\mathcal{C}_n$ as the union of all $(k-1)$-dimensional slices passing through $(1,\dots,1)$.\\
Next, we say that the box $\mathcal{C}_n$ is \emph{frameable} for the configuration $\eta \in \{0,1\}^{\mathcal{C}_n}$ if $\eta$ is connected by legal KA-$k$f transitions to a configuration for which the $k^{\textup{th}}$ frame of $\mathcal{C}_n$ is empty.
\end{defn}


\begin{defn}
Given a configuration $\eta$, we say that a box $B$ is $(d,k)$-\emph{good} for $\eta$ if all $(d-1)$-dimensional slices of $B$ are $(d-1,k-1)$-frameable for all configurations $\eta'$ that differ from $\eta$ in at most one site. We also require a good box to contain one extra empty site in addition to the ones required before.
\end{defn}

Note that this definition slightly differs from \cite{MST} by requiring an additional empty site. The reason will be clarified in the proof of Lemma \ref{existenceTstepmove}. Whenever the context allows it, we shall simply say that a box is good instead of $(d,k)$-good.

\begin{ex}
A box is $(2,2)$-good if it contains at least two empty sites in each row each column, and at least one additional empty site in the box.
\end{ex}

\begin{prop}\label{claim:probofgood}
Let $d \ge k \ge 2$ and $\l$ be defined as in (\ref{def:l}). Then :
\begin{enumerate}[(i)]
    \item The probability that the box $[\l]^d$ is $(d,k)$-frameable is at least $q^{d\l^{k-1}}$. In particular, for an appropriate choice of constants in equation (\ref{def:L}), it is much larger than $1/L$.
    \item The probability that the box $[\l]^d$ is $(d,k)$-good tends to $1$ when $q \rightarrow 0$.
\end{enumerate}
\end{prop}

\begin{proof}
The probability of being frameable is bounded by the probability that the frame is already empty. Since the size of the frame is less than $d\l^{k-1}$, the first bound follows.\\
The second bound is due to \cite{BFT}. See also Proposition 3.26 in \cite{MST} and the explanation that follows.
\end{proof}

The following definitions describes blocks that contain a droplet which is able to propagate. See Figure \ref{fig:blocks_and_boxes}.

\begin{defn}
Let $i,j$ be two adjacent sites in $\zz_\l^d$ with, suppose, $j=i+(L+1)e_\alpha$.\\
A geometric path connecting $i$ and $j$ is a sequence of adjacent boxes $B_1,\dots,B_n$ in the block $i + [L]^d$ with $B_1 = i + [\l]^d$ and $B_n = j - (\l+1)e_\alpha + [\l]^d$. We say that a geometric path is $(d,k)$-\emph{super-good} for a configuration $\eta$ if :
\begin{enumerate}
    \item For all $m$, the box $B_m$ is $(d,k)$-good.
    \item At least one of the boxes in the sequence is $(d,k)$-frameable
\end{enumerate}
\end{defn}

\begin{defn}\label{def:block-connected}
Let $i \in \zz_\l^d$ and $j=i+(L+1) e_{\alpha}$. We say that $i,j$
are $(d,k)$-\textit{block-connected} for $\eta$ if the following conditions
hold :
\begin{enumerate}
\item There exist a $(d,k)$-super-good path connecting $i$ and $j$ whose length is at most $3L$.
\item All $\left(d-1\right)$-dimensional external faces of the box $i+\left[\l\right]^{d}$
that are adjacent to $i$ are $\left(d-1,k-1\right)$-good.
\item All $\left(d-1\right)$-dimensional external faces of the box $j-\left(\l+1\right)e_{1}+\left[\l+1\right]^{d}$
that are adjacent to $j$ are $\left(d-1,k-1\right)$-good.
\end{enumerate}
\end{defn}

We stress that being block-connected does not depend on the values of $\eta\left(i\right)$
and $\eta\left(j\right)$.

This notion of being block-connected defines a percolation process on $\zz_\l^d$. Let $\overline{\eta}$ be the configuration on the edges of $\zz_\l^d$
that gives the value $1$ to an edge $i\sim j$ if $i$ and $j$ are
block-connected and $0$ otherwise. We denote by $\overline{\mu}$
the measure on these configurations induced by $\mu$. 

\begin{lem}
$\overline{\mu}$ is a stationary ergodic measure, that stochastically dominates a supercritical Bernoulli bond percolation, whose parameter tends to $1$ as $q$ tends to $0$.
\end{lem}

\begin{proof}
The probability that an edge is open depends only on the sites in the blocks adjacent to it. There are $2d$ such blocks, and the diameter of a block is $d$, hence the percolation process is $2d^2$-dependent.
By \cite{LSS}, it suffices to prove that the probability to be block connected tends to $1$ as $q$ tends to $0$.\\
The probability for a $(d-1)-$dimensional face of a box to be good tends to $1$ as $q$ goes to $0$ (see Proposition \ref{claim:probofgood}).
We now need to prove that the probability that a block satisfies the condition (1) is also large.

It will be convenient to restrict the super-good path that we seek to a two dimensional plane, as in \cite{MST}. We assume without loss of generality that $j=i+(L+1)e_1$.

By Proposition \ref{claim:probofgood} and large deviations for oriented percolation \cite{DS88} the following paths exist with high probability: 
\begin{enumerate}
    \item an up-right path of good boxes connecting $i + [\l^d]$ with $j+[\l]\times[L]\times[\l]^{d-2}$,
    \item an up-left path of good boxes connecting $j+[\l]^d - (\l+1) e_1$ to $i+[\l]\times[L]\times[\l]^{d-2}$.
\end{enumerate}
This provides a good path of boxes of length at most $3L$.

It is thus left to show that one of these boxes is super-good. This is a consequence of Proposition \ref{claim:probofgood} and the FKG inequality (since both being good and being frameable are increasing events).
\end{proof}

Following \cite{BT18}, we compare the KA dynamics to a simple random walk
on the infinite connected cluster of $\zz_\l^d$, conditioned
on the event that $0$ is in this cluster. We will denote $\overline{\mu}^{*}(\cdot)=\overline{\mu}\left(\cdot|0\leftrightarrow\infty\right)$.
It is shown in \cite{F08, DFGW} that this dynamics has a diffusive limit, given by a strictly
positive diffusion matrix.

\begin{prop}\label{Daux}
Let $D_{\text{aux}}$ be the matrix characterized by 
\[
u\cdot D_{\text{aux}}u=\inf_{f}\left\{ \sum_{\stackrel{i\in \zz_\l^d}{i\sim0}}\overline{\mu}^{*}\left(\overline{\eta}_{0i}\left[u\cdot i+f\left(\tau_{i}\overline{\eta}\right)-f(\overline{\eta})\right]^{2}\right)\right\} 
\]
for any $u\in\rr^{d}$, where the infimum is taken over local functions
$f$ on $\left\{ 0,1\right\} ^{\mathcal{E}\left(\zz_\l^d\right)}$.
Then $D_{\text{aux}}$ is bounded away from $0$ uniformly in $q$.
\end{prop}

Thus, for the rest of this section we will concentrate on proving
an inequality of the form:
\begin{equation} \label{comparaison}
D(q) \ge \Delta(q)^{-1} \,e_1\cdot D_{\text{aux}}e_1,
\end{equation}
for $\Delta(q)$ proportional to the bound to the diffusion coefficient in Theorem \ref{mainthm}.\\

First, in order to compare the KA dynamics and the auxiliary one,
we should put them on the same space.
\begin{lem}
Fix $u\in \rr^d$. Then
\[
u\cdot D_{\text{aux}}u\le\overline{\mu}\left(0\leftrightarrow\infty\right)^{-1}\,\inf_{f}\left\{ \sum_{\stackrel{i\in\zz_\l^d}{i\sim0}}\mu_{0}\left(\overline{\eta}_{0i}\left[u\cdot i+f\left(\tau_{i}\eta^{0,i}\right)-f\left(\eta\right)\right]^{2}\right)\right\}, 
\]
where now the infimum is taken over local functions
on $\Omega$.
\end{lem}

\begin{proof}
The proof follows the exact same steps as that of \cite[Lemma 5.2]{BT18}.
\end{proof}
The comparison of both dynamics will be via a path argument. We will
use the construction of \cite{MST}, by concatenating basic moves. The next definition describes a sequence of legal KA transitions which keeps track of the configuration $\eta$ and the position of a marked particle $z$.

\begin{defn}
Given $\mathcal{M} \subset \Omega$, a $T$-\emph{step move} $M$ with domain $\textup{Dom}(M) = \{(\eta,z) | \, \eta \in \mathcal{M} \, \textup{and} \, \eta(z) = 1 \}$ is a function from $\textup{Dom}(M)$ to $(\Omega \times V)^{T+1}$ such that for any $(\eta,z) \in \textup{Dom}(M)$  the sequence $(M_t\eta, z_t) := M(\eta, z)_t$ satisfies:
\begin{enumerate}
\item $M_0\eta = \eta$ and $z_0 = z$,
\item for any $t \in [T]$, the configurations $M_{t-1}\eta$ and $M_t \eta$ are either identical or linked by a legal KA transition contained in $V$,
\item for any $t \in [T]$, $z_t$ is the new position of the particle that was at site $z_{t-1}$ in $M_t \eta$. 
\end{enumerate}
For $t\in[T]$ and $\eta$, whenever $M_{t-1}\eta \ne M_t \eta$, a particle has jumped from a site to a neighbor. We denote by $x_t(\eta, z)$ (resp. $y_t(\eta,z)$) the initial (resp. final) position of the particle during the jump. More precisely, $x_t$ and $y_t$ are such that $M_{t-1}\eta(x_t) = M_t \eta (y_t) = 1$ and $M_{t-1}\eta(y_t) = M_t \eta (x_t) = 0$.\\
If $M_{t-1}\eta = M_t \eta $, we set $x_t=y_t=0$.
We denote $\text{Dom}_0(M)= \{ \eta \in \mathcal{M} \ | \ \eta(0) = 1\}$.
\end{defn}

\begin{defn}
Given a $T$-step move $M$, its \emph{information loss} $\textup{Loss}(M)$ is defined as
$$2^{\textup{Loss}(M)} = \underset{\stackrel{\eta' \in \text{Dom}(M)}{t\in [T]}}{\sup}  \# \{\eta \in \textup{Dom}(M)|M_{t-1}\eta = M_{t-1} \eta', M_t\eta = M_t\eta' \} $$
\end{defn}

The next lemma constructs a $T$-step move that exchanges a marked particle at the origin with a block-connected site.

\begin{lem}\label{existenceTstepmove}
There exists a $T$-step move $M$ satisfying the following conditions:
\begin{enumerate}
\item $\text{Dom}_0(M)=\left\{ \overline{\eta}_{0,(L+1)e_1}=1 \right\}$,
\item for all $\eta \in \text{Dom}_0(M), M(\eta,0)_T = (\eta^{0,(L+1)e_1}, (L+1)e_1)$,
\item $T\le CL\l^{\lambda}$ for $k=2$, and $T\le CL2^{\l^{d}}$ for $k\ge3$,
\item $\text{Loss}(M)\le C\log_{2}(\l)\,\l$ for $k = 2$, and $\text{Loss}(M)\le C\l^{d-1}$ for $k \ge 3$.
\end{enumerate}
\end{lem}

\begin{proof}

The proof is based on the construction of \cite{MST}, explained in \cite[Proposition 5.2.41]{SphD}. In that proposition, the $T$-step move is propagating a site through a super-good path. Note that propagating simply means that the values of the configuration in the initial and final sites are being swapped. This does not mean that the marked particle initially at $0$ reached the site $(L+1)e_1$, since in \cite{SphD}, the \emph{permutation move} exchanges any two sites even if they are both occupied (see \cite[Proposition 5.2.35]{SphD}), which is not a legal KA transition.

We will show here briefly the idea of the proof when $k=d=2$, repeating the construction of \cite{MST} with the appropriate adaptations. We will then explain how these adaptations apply to general $k,d$.

In order to construct the move $M$, we will construct a sequence of shorter, simpler, moves.
The first of them is the \emph{column exchange move}:
\begin{claim} \label{claim:col_exchanged}
Fix $y\in \mathbb{Z}^2$, and consider the configurations in which the column $y+\{0\}\times [\l]$ is empty, and the column to its right contains at least one empty site. Then there exists a $T$-step move $M$ whose domain consist of these configurations, and in the final state $M_T\eta$ the columns $y+\{0\}\times [\l]$ and $y+\{1\}\times [\l]$ are exchanged. Moreover, $\textup{Loss}(M) = O(\log_2\l)$ and $T = O(\l)$. See Figure \ref{fig:exchange_columns}.
\end{claim}

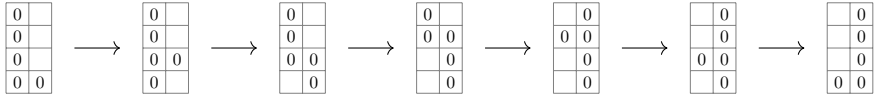
\begin{figure}
\begin{tikzpicture}[scale=0.3, every node/.style={scale=0.6}]
    \def \x{0}
    \def \y{0}	
	\draw[step=1, gray, very thin, shift={(\x,\y)}] (0,0) grid +(2,4);

	\foreach \yy in {0,1,2,3}{
		\draw[shift={(\x,\y)}] (0.5,\yy+0.5) node[black] {$0$};
	}
	\draw[shift={(\x,\y)}]  (1.5,0.5) node[black] {$0$};
		
	\draw[->,shift={(\x,\y)}]  (3,2) to (5,2);
		
    \def \x{6}
    \def \y{0}	
	\draw[step=1, gray, very thin, shift={(\x,\y)}] (0,0) grid +(2,4);

	\foreach \yy in {0,1,2,3}{
		\draw[shift={(\x,\y)}] (0.5,\yy+0.5) node[black] {$0$};
	}
	
	\draw[shift={(\x,\y)}]  (1.5,1.5) node[black] {$0$};
		
	\draw[->,shift={(\x,\y)}]  (3,2) to (5,2);
	
	\def \x{12}
    \def \y{0}	
	\draw[step=1, gray, very thin, shift={(\x,\y)}] (0,0) grid +(2,4);

	\foreach \yy in {1,2,3}{
		\draw[shift={(\x,\y)}] (0.5,\yy+0.5) node[black] {$0$};
	}
	
	\draw[shift={(\x,\y)}]  (1.5,0.5) node[black] {$0$};	
	\draw[shift={(\x,\y)}]  (1.5,1.5) node[black] {$0$};
		
	\draw[->,shift={(\x,\y)}]  (3,2) to (5,2);
	
	\def \x{18}
    \def \y{0}	
	\draw[step=1, gray, very thin, shift={(\x,\y)}] (0,0) grid +(2,4);

	\foreach \yy in {0,1,2}{
		\draw[shift={(\x,\y)}] (1.5,\yy+0.5) node[black] {$0$};
	}
	
	\draw[shift={(\x,\y)}]  (0.5,2.5) node[black] {$0$};	
	\draw[shift={(\x,\y)}]  (0.5,3.5) node[black] {$0$};
		
	\draw[->,shift={(\x,\y)}]  (3,2) to (5,2);
	
	\def \x{24}
    \def \y{0}	
	\draw[step=1, gray, very thin, shift={(\x,\y)}] (0,0) grid +(2,4);

	\foreach \yy in {0,1,2,3}{
		\draw[shift={(\x,\y)}] (1.5,\yy+0.5) node[black] {$0$};
	}
	
	\draw[shift={(\x,\y)}]  (0.5,2.5) node[black] {$0$};	
		
	\draw[->,shift={(\x,\y)}]  (3,2) to (5,2);
	
	\def \x{30}
    \def \y{0}	
	\draw[step=1, gray, very thin, shift={(\x,\y)}] (0,0) grid +(2,4);

	\foreach \yy in {0,1,2,3}{
		\draw[shift={(\x,\y)}] (1.5,\yy+0.5) node[black] {$0$};
	}
	
	\draw[shift={(\x,\y)}]  (0.5,1.5) node[black] {$0$};	
		
	\draw[->,shift={(\x,\y)}]  (3,2) to (5,2);
	
	\def \x{36}
    \def \y{0}	
	\draw[step=1, gray, very thin, shift={(\x,\y)}] (0,0) grid +(2,4);

	\foreach \yy in {0,1,2,3}{
		\draw[shift={(\x,\y)}] (1.5,\yy+0.5) node[black] {$0$};
	}
	
	\draw[shift={(\x,\y)}]  (0.5,0.5) node[black] {$0$};	
		
\end{tikzpicture}
\caption{\label{fig:exchange_columns}Exchange an empty column with a good column.}
\end{figure}

The next move we will use is the \emph{framing move}:
\begin{claim} \label{claim:framing_move}
Fix a box, and consider the configurations for which the box is good, and, in addition, its bottom row is empty. Then there exists a T-step move $M$ whose domain consists of these configurations, and in the final state $M_T\eta$ the left column is also empty. Moreover, $\textup{Loss}(M) = O(\l\log_2\l)$ and $T = O(\l^2)$. See figure \ref{fig:framing_move}.\\
Note that the marked particle could move when we frame the box. See Figure \ref{fig:framing_move}.
\end{claim}

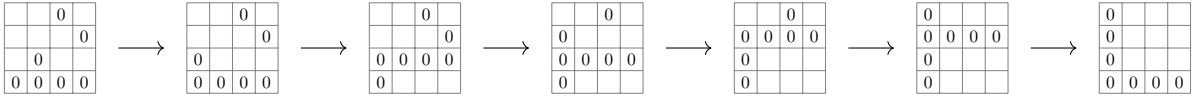
\begin{figure}
\begin{tikzpicture}[scale=0.3, every node/.style={scale=0.6}]
    \def \x{0}
    \def \y{0}	
	\draw[step=1, gray, very thin, shift={(\x,\y)}] (0,0) grid +(4,4);

	\foreach \xx in {0,1,2,3}{
		\draw[shift={(\x,\y)}] (0.5+\xx,0.5) node[black] {$0$};
	}
	\draw[shift={(\x,\y)}]  (1.5,1.5) node[black] {$0$};
	\draw[shift={(\x,\y)}]  (3.5,2.5) node[black] {$0$};
	\draw[shift={(\x,\y)}]  (2.5,3.5) node[black] {$0$};
		
	\draw[->,shift={(\x,\y)}]  (5,2) to +(2,0);
		
    \def \x{8}
    \def \y{0}	
	\draw[step=1, gray, very thin, shift={(\x,\y)}] (0,0) grid +(4,4);

	\foreach \xx in {0,1,2,3}{
		\draw[shift={(\x,\y)}] (0.5+\xx,0.5) node[black] {$0$};
	}
	\draw[shift={(\x,\y)}]  (0.5,1.5) node[black] {$0$};
	\draw[shift={(\x,\y)}]  (3.5,2.5) node[black] {$0$};
	\draw[shift={(\x,\y)}]  (2.5,3.5) node[black] {$0$};
		
	\draw[->,shift={(\x,\y)}]  (5,2) to +(2,0);
	
	\def \x{16}
    \def \y{0}	
	\draw[step=1, gray, very thin, shift={(\x,\y)}] (0,0) grid +(4,4);

	\foreach \xx in {0,1,2,3}{
		\draw[shift={(\x,\y)}] (0.5+\xx,1.5) node[black] {$0$};
	}
	\draw[shift={(\x,\y)}]  (0.5,0.5) node[black] {$0$};
	\draw[shift={(\x,\y)}]  (3.5,2.5) node[black] {$0$};
	\draw[shift={(\x,\y)}]  (2.5,3.5) node[black] {$0$};
		
	\draw[->,shift={(\x,\y)}]  (5,2) to +(2,0);
	
	\def \x{24}
    \def \y{0}	
	\draw[step=1, gray, very thin, shift={(\x,\y)}] (0,0) grid +(4,4);

	\foreach \xx in {0,1,2,3}{
		\draw[shift={(\x,\y)}] (0.5+\xx,1.5) node[black] {$0$};
	}
	\draw[shift={(\x,\y)}]  (0.5,0.5) node[black] {$0$};
	\draw[shift={(\x,\y)}]  (0.5,2.5) node[black] {$0$};
	\draw[shift={(\x,\y)}]  (2.5,3.5) node[black] {$0$};
		
	\draw[->,shift={(\x,\y)}]  (5,2) to +(2,0);
	
	\def \x{32}
    \def \y{0}	
	\draw[step=1, gray, very thin, shift={(\x,\y)}] (0,0) grid +(4,4);

	\foreach \xx in {0,1,2,3}{
		\draw[shift={(\x,\y)}] (0.5+\xx,2.5) node[black] {$0$};
	}
	\draw[shift={(\x,\y)}]  (0.5,0.5) node[black] {$0$};
	\draw[shift={(\x,\y)}]  (0.5,1.5) node[black] {$0$};
	\draw[shift={(\x,\y)}]  (2.5,3.5) node[black] {$0$};
		
	\draw[->,shift={(\x,\y)}]  (5,2) to +(2,0);
	
	\def \x{40}
    \def \y{0}	
	\draw[step=1, gray, very thin, shift={(\x,\y)}] (0,0) grid +(4,4);

	\foreach \xx in {0,1,2,3}{
		\draw[shift={(\x,\y)}] (0.5+\xx,2.5) node[black] {$0$};
	}
	\draw[shift={(\x,\y)}]  (0.5,0.5) node[black] {$0$};
	\draw[shift={(\x,\y)}]  (0.5,1.5) node[black] {$0$};
	\draw[shift={(\x,\y)}]  (0.5,3.5) node[black] {$0$};
		
	\draw[->,shift={(\x,\y)}]  (5,2) to +(2,0);
	
	\def \x{48}
    \def \y{0}	
	\draw[step=1, gray, very thin, shift={(\x,\y)}] (0,0) grid +(4,4);

	\foreach \xx in {0,1,2,3}{
		\draw[shift={(\x,\y)}] (0.5+\xx,0.5) node[black] {$0$};
	}
	\draw[shift={(\x,\y)}]  (0.5,2.5) node[black] {$0$};
	\draw[shift={(\x,\y)}]  (0.5,1.5) node[black] {$0$};
	\draw[shift={(\x,\y)}]  (0.5,3.5) node[black] {$0$};
		
\end{tikzpicture}
\caption{\label{fig:framing_move}Framing move.}
\end{figure}

When a box is framed, we are able to permute its sites:
\begin{claim} \label{claim:permutation_move}
Consider the box $[\l]^d$, and fix any permutation $\sigma$ of the sites $[2,\l]^d$, which by convention fixes the sites outside $[2,\l]^d$. Consider the configurations for which $[\l]^d$ is framed and $[2,\l]^d$ contain at least one empty site. Then there exists a $T$-step move $M$ whose domain consists of these configurations, and in the final state $M_T(\eta)$ the sites are permuted according to $\sigma$, i.e., $(M_T \eta)(\sigma x) = \eta(x)$ for all $x\in[2,\l]^d$. Moreover the marked particle moves from its original position $X$ to $\sigma X$. See figure \ref{fig:permutation}.
\end{claim}

\begin{figure}
\begin{tikzpicture}[scale=0.3, every node/.style={scale=0.6}]
	
	\def \x{0}
    \def \y{0}	
	\draw[step=1, gray, very thin, shift={(\x,\y)}] (0,0) grid +(4,4);

	\foreach \xx in {0,1,2,3}{
		\draw[shift={(\x,\y)}] (0.5+\xx,0.5) node[black] {$0$};
	}
	\draw[shift={(\x,\y)}]  (0.5,2.5) node[black] {$0$};
	\draw[shift={(\x,\y)}]  (0.5,1.5) node[black] {$0$};
	\draw[shift={(\x,\y)}]  (0.5,3.5) node[black] {$0$};
	
	\draw[shift={(\x,\y)}]  (3.5,2.5) node[black] {$b$};
	\draw[shift={(\x,\y)}]  (1.5,2.5) node[RoyalBlue] {$0$};
	\draw[shift={(\x,\y)}]  (2.5,2.5) node[black] {$a$};
		
	\draw[->,shift={(\x,\y)}]  (5,2) to +(2,0);
	
	\def \x{8}
    \def \y{0}	
	\draw[step=1, gray, very thin, shift={(\x,\y)}] (0,0) grid +(4,4);

	\foreach \xx in {0,1,2,3}{
		\draw[shift={(\x,\y)}] (0.5+\xx,1.5) node[black] {$0$};
	}
	\draw[shift={(\x,\y)}]  (0.5,2.5) node[black] {$0$};
	\draw[shift={(\x,\y)}]  (0.5,0.5) node[black] {$0$};
	\draw[shift={(\x,\y)}]  (0.5,3.5) node[black] {$0$};
	
	\draw[shift={(\x,\y)}]  (3.5,2.5) node[black] {$b$};
	\draw[shift={(\x,\y)}]  (1.5,2.5) node[RoyalBlue] {$0$};
	\draw[shift={(\x,\y)}]  (2.5,2.5) node[black] {$a$};
		
	\draw[->,shift={(\x,\y)}]  (5,2) to +(2,0);
	
	\def \x{16}
    \def \y{0}	
	\draw[step=1, gray, very thin, shift={(\x,\y)}] (0,0) grid +(4,4);

	\foreach \xx in {0,1,2,3}{
		\draw[shift={(\x,\y)}] (0.5+\xx,1.5) node[black] {$0$};
	}
	\draw[shift={(\x,\y)}]  (0.5,2.5) node[black] {$0$};
	\draw[shift={(\x,\y)}]  (0.5,0.5) node[black] {$0$};
	\draw[shift={(\x,\y)}]  (0.5,3.5) node[black] {$0$};
	
	\draw[shift={(\x,\y)}]  (3.5,2.5) node[black] {$b$};
	\draw[shift={(\x,\y)}]  (1.5,2.5) node[black] {$a$};
	\draw[shift={(\x,\y)}]  (2.5,2.5) node[RoyalBlue] {$0$};
		
	\draw[->,shift={(\x,\y)}]  (5,2) to +(2,0);
	
	\def \x{24}
    \def \y{0}	
	\draw[step=1, gray, very thin, shift={(\x,\y)}] (0,0) grid +(4,4);

	\foreach \xx in {0,1,2,3}{
		\draw[shift={(\x,\y)}] (0.5+\xx,1.5) node[black] {$0$};
	}
	\draw[shift={(\x,\y)}]  (0.5,2.5) node[black] {$0$};
	\draw[shift={(\x,\y)}]  (0.5,0.5) node[black] {$0$};
	\draw[shift={(\x,\y)}]  (0.5,3.5) node[black] {$0$};
	
	\draw[shift={(\x,\y)}]  (3.5,2.5) node[RoyalBlue] {$0$};
	\draw[shift={(\x,\y)}]  (1.5,2.5) node[black] {$a$};
	\draw[shift={(\x,\y)}]  (2.5,2.5) node[black] {$b$};
		
	\draw[->,shift={(\x,\y)}]  (5,2) to +(2,0);
	
	\def \x{32}
    \def \y{0}	
	\draw[step=1, gray, very thin, shift={(\x,\y)}] (0,0) grid +(4,4);

	\foreach \xx in {0,1,2,3}{
		\draw[shift={(\x,\y)}] (0.5+\xx,0.5) node[black] {$0$};
	}
	\draw[shift={(\x,\y)}]  (0.5,2.5) node[black] {$0$};
	\draw[shift={(\x,\y)}]  (0.5,1.5) node[black] {$0$};
	\draw[shift={(\x,\y)}]  (0.5,3.5) node[black] {$0$};
	
	\draw[shift={(\x,\y)}]  (3.5,2.5) node[RoyalBlue] {$0$};
	\draw[shift={(\x,\y)}]  (1.5,2.5) node[black] {$a$};
	\draw[shift={(\x,\y)}]  (2.5,2.5) node[black] {$b$};
		
	\draw[->,shift={(\x,\y)}]  (5,2) to +(2,0);
	
	\def \x{40}
    \def \y{0}	
	\draw[step=1, gray, very thin, shift={(\x,\y)}] (0,0) grid +(4,4);

	\foreach \xx in {0,1,2,3}{
		\draw[shift={(\x,\y)}] (0.5+\xx,0.5) node[black] {$0$};
	}
	\draw[shift={(\x,\y)}]  (0.5,2.5) node[black] {$0$};
	\draw[shift={(\x,\y)}]  (0.5,1.5) node[black] {$0$};
	\draw[shift={(\x,\y)}]  (0.5,3.5) node[black] {$0$};
	
	\draw[shift={(\x,\y)}]  (3.5,2.5) node[RoyalBlue] {$0$};
	\draw[shift={(\x,\y)}]  (1.5,1.5) node[black] {$a$};
	\draw[shift={(\x,\y)}]  (2.5,2.5) node[black] {$b$};
		
	\draw[->,shift={(\x,\y)}]  (5,2) to +(2,0);
	
	\def \x{48}
    \def \y{0}	
	\draw[step=1, gray, very thin, shift={(\x,\y)}] (0,0) grid +(4,4);

	\foreach \xx in {0,1,2,3}{
		\draw[shift={(\x,\y)}] (0.5+\xx,0.5) node[black] {$0$};
	}
	\draw[shift={(\x,\y)}]  (0.5,2.5) node[black] {$0$};
	\draw[shift={(\x,\y)}]  (0.5,1.5) node[black] {$0$};
	\draw[shift={(\x,\y)}]  (0.5,3.5) node[black] {$0$};
	
	\draw[shift={(\x,\y)}]  (3.5,2.5) node[RoyalBlue] {$0$};
	\draw[shift={(\x,\y)}]  (3.5,1.5) node[black] {$a$};
	\draw[shift={(\x,\y)}]  (2.5,2.5) node[black] {$b$};
		
	\draw[->,shift={(\x,\y)}]  (5,2) to +(2,0);
	
	\def \x{0}
    \def \y{-6}	
	\draw[step=1, gray, very thin, shift={(\x,\y)}] (0,0) grid +(4,4);

	\foreach \xx in {0,1,2,3}{
		\draw[shift={(\x,\y)}] (0.5+\xx,0.5) node[black] {$0$};
	}
	\draw[shift={(\x,\y)}]  (2.5,2.5) node[black] {$0$};
	\draw[shift={(\x,\y)}]  (2.5,1.5) node[black] {$0$};
	\draw[shift={(\x,\y)}]  (2.5,3.5) node[black] {$0$};
	
	\draw[shift={(\x,\y)}]  (3.5,2.5) node[RoyalBlue] {$0$};
	\draw[shift={(\x,\y)}]  (3.5,1.5) node[black] {$a$};
	\draw[shift={(\x,\y)}]  (1.5,2.5) node[black] {$b$};
		
	\draw[->,shift={(\x,\y)}]  (5,2) to +(2,0);
	
	\def \x{8}
    \def \y{-6}	
	\draw[step=1, gray, very thin, shift={(\x,\y)}] (0,0) grid +(4,4);

	\foreach \xx in {0,1,2,3}{
		\draw[shift={(\x,\y)}] (0.5+\xx,0.5) node[black] {$0$};
	}
	\draw[shift={(\x,\y)}]  (2.5,2.5) node[black] {$0$};
	\draw[shift={(\x,\y)}]  (2.5,1.5) node[black] {$0$};
	\draw[shift={(\x,\y)}]  (2.5,3.5) node[black] {$0$};
	
	\draw[shift={(\x,\y)}]  (3.5,2.5) node[black] {$a$};
	\draw[shift={(\x,\y)}]  (3.5,1.5) node[RoyalBlue] {$0$};
	\draw[shift={(\x,\y)}]  (1.5,2.5) node[black] {$b$};
		
	\draw[->,shift={(\x,\y)}]  (5,2) to +(2,0);
	
	\def \x{16}
    \def \y{-6}	
	\draw[step=1, gray, very thin, shift={(\x,\y)}] (0,0) grid +(4,4);

	\foreach \xx in {0,1,2,3}{
		\draw[shift={(\x,\y)}] (0.5+\xx,0.5) node[black] {$0$};
	}
	\draw[shift={(\x,\y)}]  (0.5,2.5) node[black] {$0$};
	\draw[shift={(\x,\y)}]  (0.5,1.5) node[black] {$0$};
	\draw[shift={(\x,\y)}]  (0.5,3.5) node[black] {$0$};
	
	\draw[shift={(\x,\y)}]  (3.5,2.5) node[black] {$a$};
	\draw[shift={(\x,\y)}]  (3.5,1.5) node[RoyalBlue] {$0$};
	\draw[shift={(\x,\y)}]  (2.5,2.5) node[black] {$b$};
		
	\draw[->,shift={(\x,\y)}]  (5,2) to +(2,0);
	
	\def \x{24}
    \def \y{-6}	
	\draw[step=1, gray, very thin, shift={(\x,\y)}] (0,0) grid +(4,4);

	\foreach \xx in {0,1,2,3}{
		\draw[shift={(\x,\y)}] (0.5+\xx,0.5) node[black] {$0$};
	}
	\draw[shift={(\x,\y)}]  (0.5,2.5) node[black] {$0$};
	\draw[shift={(\x,\y)}]  (0.5,1.5) node[black] {$0$};
	\draw[shift={(\x,\y)}]  (0.5,3.5) node[black] {$0$};
	
	\draw[shift={(\x,\y)}]  (3.5,2.5) node[black] {$a$};
	\draw[shift={(\x,\y)}]  (1.5,2.5) node[RoyalBlue] {$0$};
	\draw[shift={(\x,\y)}]  (2.5,2.5) node[black] {$b$};
	
\end{tikzpicture}
\caption{\label{fig:permutation}Permutation move, exchanging $a$ with $b$. Note that if $a$ is $1$ and $b$ is the marked particle, we are not allowed to exchange them directly, which is the reason we introduce the blue $0$.}
\end{figure}
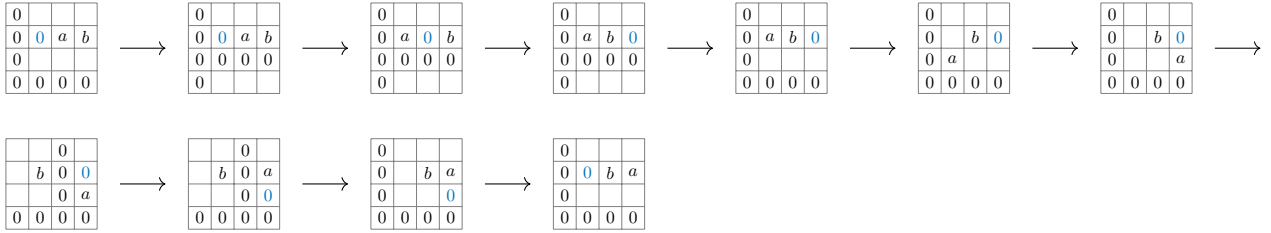

One last ingredient before constructing the large move $M$ is the \emph{jump move}, that will allow us to hop a  particle over a row of empty sites:
\begin{claim} \label{claim:jump_move}
Fix $y \in \zz^d$ and some site $\star \in y + \{-1\}\times[\l]$. Consider the configurations in which the column $y + \{0\} \times [\l]$ is empty, the column $y + \{-1\} \times [\l]$ contains at least one empty site not counting $\star$, and the column $y + \{1\} \times [\l]$ contains at least two empty sites. Then there exists a $T$-step move $M$ whose domain consists of these configurations, and in the final state, $M_T(\eta)$ the sites $\star$ and $\star + (2,0)$ are exchanged, and $z_T(\eta,\star) = \star + (2,0)$. Moreover, $\text{Loss}(M) = O(\log_2(\l))$ and $T = O(\l)$. See Figure \ref{fig:jump_move}.
\end{claim}

\begin{figure}
\begin{tikzpicture}[scale=0.3, every node/.style={scale=0.6}]
	\def \x{0};
    \def \y{0};
    
	\draw[step=1, gray, very thin, shift={(\x,\y)}] (0,0) grid +(3,4);

	\foreach \yy in {0,1,2,3}{
		\draw[shift={(\x,\y)}] (1.5,\yy+0.5) node[black] {$0$};
	}
	
	\draw[shift={(\x,\y)}]  (0.5,2.5) node[black] {$\star$};
	\draw[shift={(\x,\y)}]  (2.5,2.5) node[black] {$a$};
		
	\draw[->,shift={(\x,\y)}]  (4,2) to +(2,0);
	
	\def \x{7};
    \def \y{0};
    
	\draw[step=1, gray, very thin, shift={(\x,\y)}] (0,0) grid +(3,4);

	\foreach \yy in {0,1,2,3}{
		\draw[shift={(\x,\y)}] (0.5,\yy+0.5) node[black] {$0$};
	}
	
	\draw[shift={(\x,\y)}] (1.5,0.5) node[black] {$0$};	
	\draw[shift={(\x,\y)}] (2.5,0.5) node[black] {$0$};	
	
	\draw[shift={(\x,\y)}] (1.5,3.5) node[black] {$\star$};
	\draw[shift={(\x,\y)}] (2.5,2.5) node[black] {$a$};
		
	\draw[->,shift={(\x,\y)}]  (4,2) to +(2,0);
	
	\def \x{14};
    \def \y{0};
    
	\draw[step=1, gray, very thin, shift={(\x,\y)}] (0,0) grid +(3,4);

	\foreach \yy in {0,1,2,3}{
		\draw[shift={(\x,\y)}] (0.5,\yy+0.5) node[black] {$0$};
	}
	
	\draw[shift={(\x,\y)}] (1.5,0.5) node[black] {$0$};	
	\draw[shift={(\x,\y)}] (2.5,0.5) node[black] {$0$};	
	
	\draw[shift={(\x,\y)}] (1.5,3.5) node[black] {$a$};
	\draw[shift={(\x,\y)}] (2.5,2.5) node[black] {$\star$};
		
	\draw[->,shift={(\x,\y)}]  (4,2) to +(2,0);
	
	\def \x{21};
    \def \y{0};
    
	\draw[step=1, gray, very thin, shift={(\x,\y)}] (0,0) grid +(3,4);

	\foreach \yy in {0,1,2,3}{
		\draw[shift={(\x,\y)}] (1.5,\yy+0.5) node[black] {$0$};
	}
	
	\draw[shift={(\x,\y)}]  (0.5,2.5) node[black] {$a$};
	\draw[shift={(\x,\y)}]  (2.5,2.5) node[black] {$\star$};

\end{tikzpicture}
\caption{\label{fig:jump_move}Jump move. We start with exchanging the middle column with the right one, and then use proposition something to frame the box (during which $\star$ and $a$ may move). Then we use the permutation move and go back.}
\end{figure}
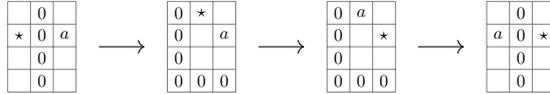

We are now ready to construct the move as shown in Figure \ref{fig:exchangepath}, exchanging the marked particle initially at $0$ with the particle/vacancy $\star$ initially at $(L+1)e$. First we use Claim \ref{claim:col_exchanged} in order to propagate the empty column, then we frame the box $[1,\l]\times[0,\l]$ using Claim \ref{claim:framing_move}. We can then frame the column $\{1\}\times[0,\l]$, and then use the permutation move (with the modification described above) in order to bring the marked particle to the position $(\l-1,\l-1)$. We then use the jump move (Claim \ref{claim:jump_move}), and clean up the modifications to the box $[0,\l]^2$. Using again Claim \ref{claim:col_exchanged}, we can move the marked particle to the bottom right box, and apply the same framing procedure as before in order to exchange it with $\star$ and move $\star$ to $(0,0)$. All that is left is to take the row of $0$s back to its original position.

For general $k,d$ the same proof as \cite{MST} will allow us to construct $M$, where the only modification is in the definition of the permutation move, which now takes into account the position of the marked particle. This is done in the exact same way as we have seen in Claim \ref{claim:permutation_move} for the case $k=d=2$.

 \end{proof}

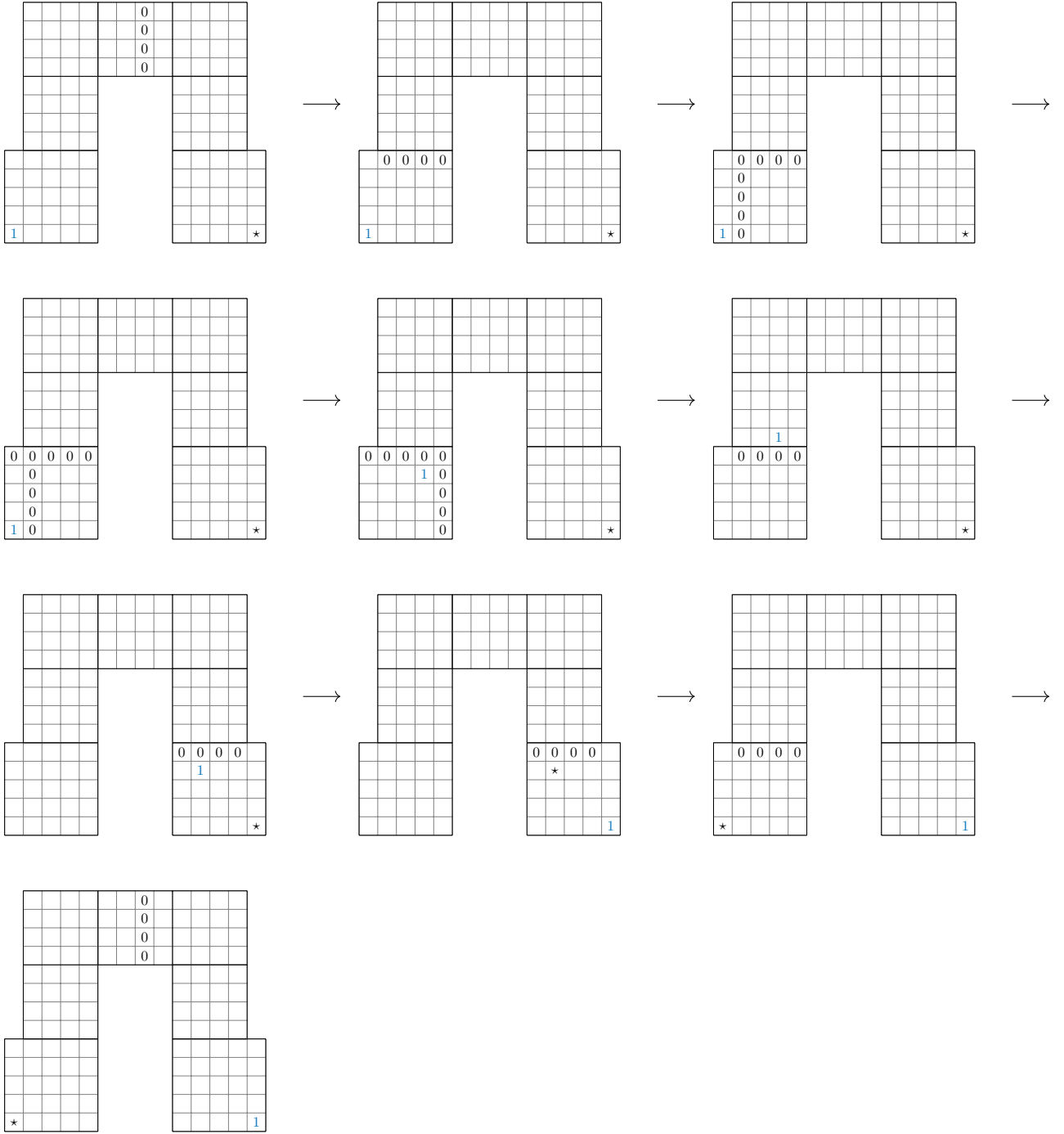
\begin{figure}[h]
\begin{tikzpicture}[scale=0.3, every node/.style={scale=0.6}]
	
	\def\xx{0}
	\def\yy{0}
	
	\draw[step=1, gray, very thin, xshift=\xx cm, yshift = \yy cm] (0,0) grid (5,5);
	\draw[step=5, black, very thin, xshift=\xx cm, yshift = \yy cm] (0,0) grid (5,5);
	\draw[step=1, gray, very thin,xshift=\xx+1 cm, yshift = \yy+5 cm] (0,0) grid (4,8);
	\draw[step=4, black, very thin,xshift=\xx+1 cm, yshift = \yy+5 cm] (0,0) grid (4,8);
	\draw[step=1, gray, very thin,xshift=\xx+5 cm, yshift = \yy+9 cm] (0,0) grid (8,4);
	\draw[step=4, black, very thin,xshift=\xx+5 cm, yshift = \yy+9 cm] (0,0) grid (8,4);
	\draw[step=1, gray, very thin,xshift=\xx+9 cm, yshift = \yy+5 cm] (0,0) grid (4,4);
	\draw[step=4, black, very thin,xshift=\xx+9 cm, yshift = \yy+5 cm] (0,0) grid (4,4);
	\draw[step=1, gray, very thin,xshift=\xx+9 cm, yshift = \yy cm] (0,0) grid (5,5);
	\draw[step=5, black, very thin,xshift=\xx+9 cm, yshift = \yy cm] (0,0) grid (5,5);
	
    \foreach \y in {0,...,3}{
		\draw (\xx+7.5,\yy+\y+9.5) node[black] {$0$};
	}
    \draw (\xx+0.5,\yy+0.5) node[RoyalBlue] {$1$};
    \draw (\xx+13.5,\yy+0.5) node[black] {$\star$};
    
    \draw[->]  (\xx+16,\yy+7.5) to (\xx+18,\yy+7.5);
    
    \def\xx{19}
    
	\draw[step=1, gray, very thin, xshift=\xx cm, yshift = \yy cm] (0,0) grid (5,5);
	\draw[step=5, black, very thin, xshift=\xx cm, yshift = \yy cm] (0,0) grid (5,5);
	\pgfmathsetmacro{\x}{\xx+1}
	\pgfmathsetmacro{\y}{\yy+5}
	\draw[step=1, gray, very thin,xshift=\x cm, yshift = \y cm] (0,0) grid (4,8);
	\draw[step=4, black, very thin,xshift=\x cm, yshift = \y cm] (0,0) grid (4,8);
	\pgfmathsetmacro{\x}{\xx+5}
	\pgfmathsetmacro{\y}{\yy+9}
	\draw[step=1, gray, very thin,xshift=\x cm, yshift = \y cm] (0,0) grid (8,4);
	\draw[step=4, black, very thin,xshift=\x cm, yshift = \y cm] (0,0) grid (8,4);
	\pgfmathsetmacro{\x}{\xx+9}
	\pgfmathsetmacro{\y}{\yy+5}
	\draw[step=1, gray, very thin,xshift=\x cm, yshift = \y cm] (0,0) grid (4,4);
	\draw[step=4, black, very thin,xshift=\x cm, yshift = \y cm] (0,0) grid (4,4);
	\pgfmathsetmacro{\x}{\xx+9}
	\pgfmathsetmacro{\y}{\yy+0}
	\draw[step=1, gray, very thin,xshift=\x cm, yshift = \y cm] (0,0) grid (5,5);
	\draw[step=5, black, very thin,xshift=\x cm, yshift = \y cm] (0,0) grid (5,5);
	
    \foreach \x in {0,...,3}{
		\draw (\xx+\x+1.5,\yy+4.5) node[black] {$0$};
	}
    \draw (\xx+0.5,\yy+0.5) node[RoyalBlue] {$1$};
    \draw (\xx+13.5,\yy+0.5) node[black] {$\star$};
    
    \draw[->]  (\xx+16,\yy+7.5) to (\xx+18,\yy+7.5);
    
    \def\xx{38}
    
    \draw[step=1, gray, very thin, xshift=\xx cm, yshift = \yy cm] (0,0) grid (5,5);
	\draw[step=5, black, very thin, xshift=\xx cm, yshift = \yy cm] (0,0) grid (5,5);
	\pgfmathsetmacro{\x}{\xx+1}
	\pgfmathsetmacro{\y}{\yy+5}
	\draw[step=1, gray, very thin,xshift=\x cm, yshift = \y cm] (0,0) grid (4,8);
	\draw[step=4, black, very thin,xshift=\x cm, yshift = \y cm] (0,0) grid (4,8);
	\pgfmathsetmacro{\x}{\xx+5}
	\pgfmathsetmacro{\y}{\yy+9}
	\draw[step=1, gray, very thin,xshift=\x cm, yshift = \y cm] (0,0) grid (8,4);
	\draw[step=4, black, very thin,xshift=\x cm, yshift = \y cm] (0,0) grid (8,4);
	\pgfmathsetmacro{\x}{\xx+9}
	\pgfmathsetmacro{\y}{\yy+5}
	\draw[step=1, gray, very thin,xshift=\x cm, yshift = \y cm] (0,0) grid (4,4);
	\draw[step=4, black, very thin,xshift=\x cm, yshift = \y cm] (0,0) grid (4,4);
	\pgfmathsetmacro{\x}{\xx+9}
	\pgfmathsetmacro{\y}{\yy+0}
	\draw[step=1, gray, very thin,xshift=\x cm, yshift = \y cm] (0,0) grid (5,5);
	\draw[step=5, black, very thin,xshift=\x cm, yshift = \y cm] (0,0) grid (5,5);
	
    \foreach \x in {1,...,4}{
		\draw (\xx+\x+0.5,\yy+4.5) node[black] {$0$};
	}
	\foreach \y in {0,...,3}{
		\draw (\xx+1.5,\yy+\y+0.5) node[black] {$0$};
	}
    \draw (\xx+0.5,\yy+0.5) node[RoyalBlue] {$1$};
    \draw (\xx+13.5,\yy+0.5) node[black] {$\star$};
    
    \draw[->]  (\xx+16,\yy+7.5) to (\xx+18,\yy+7.5);
    
    \def\xx{0}
    \def\yy{-16}
    
     \draw[step=1, gray, very thin, xshift=\xx cm, yshift = \yy cm] (0,0) grid (5,5);
	\draw[step=5, black, very thin, xshift=\xx cm, yshift = \yy cm] (0,0) grid (5,5);
	\pgfmathsetmacro{\x}{\xx+1}
	\pgfmathsetmacro{\y}{\yy+5}
	\draw[step=1, gray, very thin,xshift=\x cm, yshift = \y cm] (0,0) grid (4,8);
	\draw[step=4, black, very thin,xshift=\x cm, yshift = \y cm] (0,0) grid (4,8);
	\pgfmathsetmacro{\x}{\xx+5}
	\pgfmathsetmacro{\y}{\yy+9}
	\draw[step=1, gray, very thin,xshift=\x cm, yshift = \y cm] (0,0) grid (8,4);
	\draw[step=4, black, very thin,xshift=\x cm, yshift = \y cm] (0,0) grid (8,4);
	\pgfmathsetmacro{\x}{\xx+9}
	\pgfmathsetmacro{\y}{\yy+5}
	\draw[step=1, gray, very thin,xshift=\x cm, yshift = \y cm] (0,0) grid (4,4);
	\draw[step=4, black, very thin,xshift=\x cm, yshift = \y cm] (0,0) grid (4,4);
	\pgfmathsetmacro{\x}{\xx+9}
	\pgfmathsetmacro{\y}{\yy+0}
	\draw[step=1, gray, very thin,xshift=\x cm, yshift = \y cm] (0,0) grid (5,5);
	\draw[step=5, black, very thin,xshift=\x cm, yshift = \y cm] (0,0) grid (5,5);
	
    \foreach \x in {0,...,4}{
		\draw (\xx+\x+0.5,\yy+4.5) node[black] {$0$};
	}
	\foreach \y in {0,...,3}{
		\draw (\xx+1.5,\yy+\y+0.5) node[black] {$0$};
	}
	
    \draw (\xx+0.5,\yy+0.5) node[RoyalBlue] {$1$};
    \draw (\xx+13.5,\yy+0.5) node[black] {$\star$};
    
    \draw[->]  (\xx+16,\yy+7.5) to (\xx+18,\yy+7.5);
    
    \def\xx{19}
    
    \draw[step=1, gray, very thin, xshift=\xx cm, yshift = \yy cm] (0,0) grid (5,5);
	\draw[step=5, black, very thin, xshift=\xx cm, yshift = \yy cm] (0,0) grid (5,5);
	\pgfmathsetmacro{\x}{\xx+1}
	\pgfmathsetmacro{\y}{\yy+5}
	\draw[step=1, gray, very thin,xshift=\x cm, yshift = \y cm] (0,0) grid (4,8);
	\draw[step=4, black, very thin,xshift=\x cm, yshift = \y cm] (0,0) grid (4,8);
	\pgfmathsetmacro{\x}{\xx+5}
	\pgfmathsetmacro{\y}{\yy+9}
	\draw[step=1, gray, very thin,xshift=\x cm, yshift = \y cm] (0,0) grid (8,4);
	\draw[step=4, black, very thin,xshift=\x cm, yshift = \y cm] (0,0) grid (8,4);
	\pgfmathsetmacro{\x}{\xx+9}
	\pgfmathsetmacro{\y}{\yy+5}
	\draw[step=1, gray, very thin,xshift=\x cm, yshift = \y cm] (0,0) grid (4,4);
	\draw[step=4, black, very thin,xshift=\x cm, yshift = \y cm] (0,0) grid (4,4);
	\pgfmathsetmacro{\x}{\xx+9}
	\pgfmathsetmacro{\y}{\yy+0}
	\draw[step=1, gray, very thin,xshift=\x cm, yshift = \y cm] (0,0) grid (5,5);
	\draw[step=5, black, very thin,xshift=\x cm, yshift = \y cm] (0,0) grid (5,5);
	
    \foreach \x in {0,...,4}{
		\draw (\xx+\x+0.5,\yy+4.5) node[black] {$0$};
	}
	\foreach \y in {0,...,3}{
		\draw (\xx+4.5,\yy+\y+0.5) node[black] {$0$};
	}
	
    \draw (\xx+3.5,\yy+3.5) node[RoyalBlue] {$1$};
    \draw (\xx+13.5,\yy+0.5) node[black] {$\star$};
    
    \draw[->]  (\xx+16,\yy+7.5) to (\xx+18,\yy+7.5);
    
    \def\xx{38}
    
    \draw[step=1, gray, very thin, xshift=\xx cm, yshift = \yy cm] (0,0) grid (5,5);
	\draw[step=5, black, very thin, xshift=\xx cm, yshift = \yy cm] (0,0) grid (5,5);
	\pgfmathsetmacro{\x}{\xx+1}
	\pgfmathsetmacro{\y}{\yy+5}
	\draw[step=1, gray, very thin,xshift=\x cm, yshift = \y cm] (0,0) grid (4,8);
	\draw[step=4, black, very thin,xshift=\x cm, yshift = \y cm] (0,0) grid (4,8);
	\pgfmathsetmacro{\x}{\xx+5}
	\pgfmathsetmacro{\y}{\yy+9}
	\draw[step=1, gray, very thin,xshift=\x cm, yshift = \y cm] (0,0) grid (8,4);
	\draw[step=4, black, very thin,xshift=\x cm, yshift = \y cm] (0,0) grid (8,4);
	\pgfmathsetmacro{\x}{\xx+9}
	\pgfmathsetmacro{\y}{\yy+5}
	\draw[step=1, gray, very thin,xshift=\x cm, yshift = \y cm] (0,0) grid (4,4);
	\draw[step=4, black, very thin,xshift=\x cm, yshift = \y cm] (0,0) grid (4,4);
	\pgfmathsetmacro{\x}{\xx+9}
	\pgfmathsetmacro{\y}{\yy+0}
	\draw[step=1, gray, very thin,xshift=\x cm, yshift = \y cm] (0,0) grid (5,5);
	\draw[step=5, black, very thin,xshift=\x cm, yshift = \y cm] (0,0) grid (5,5);
	
    \foreach \x in {1,...,4}{
		\draw (\xx+\x+0.5,\yy+4.5) node[black] {$0$};
	}

    \draw (\xx+3.5,\yy+5.5) node[RoyalBlue] {$1$};
    \draw (\xx+13.5,\yy+0.5) node[black] {$\star$};
    
    \draw[->]  (\xx+16,\yy+7.5) to (\xx+18,\yy+7.5);
    
    \def\xx{0}
    \def\yy{-32}
    
    \draw[step=1, gray, very thin, xshift=\xx cm, yshift = \yy cm] (0,0) grid (5,5);
	\draw[step=5, black, very thin, xshift=\xx cm, yshift = \yy cm] (0,0) grid (5,5);
	\pgfmathsetmacro{\x}{\xx+1}
	\pgfmathsetmacro{\y}{\yy+5}
	\draw[step=1, gray, very thin,xshift=\x cm, yshift = \y cm] (0,0) grid (4,8);
	\draw[step=4, black, very thin,xshift=\x cm, yshift = \y cm] (0,0) grid (4,8);
	\pgfmathsetmacro{\x}{\xx+5}
	\pgfmathsetmacro{\y}{\yy+9}
	\draw[step=1, gray, very thin,xshift=\x cm, yshift = \y cm] (0,0) grid (8,4);
	\draw[step=4, black, very thin,xshift=\x cm, yshift = \y cm] (0,0) grid (8,4);
	\pgfmathsetmacro{\x}{\xx+9}
	\pgfmathsetmacro{\y}{\yy+5}
	\draw[step=1, gray, very thin,xshift=\x cm, yshift = \y cm] (0,0) grid (4,4);
	\draw[step=4, black, very thin,xshift=\x cm, yshift = \y cm] (0,0) grid (4,4);
	\pgfmathsetmacro{\x}{\xx+9}
	\pgfmathsetmacro{\y}{\yy+0}
	\draw[step=1, gray, very thin,xshift=\x cm, yshift = \y cm] (0,0) grid (5,5);
	\draw[step=5, black, very thin,xshift=\x cm, yshift = \y cm] (0,0) grid (5,5);
	
    \foreach \x in {1,...,4}{
		\draw (\xx+\x+8.5,\yy+4.5) node[black] {$0$};
	}

    \draw (\xx+10.5,\yy+3.5) node[RoyalBlue] {$1$};
    \draw (\xx+13.5,\yy+0.5) node[black] {$\star$};
    
    \draw[->]  (\xx+16,\yy+7.5) to (\xx+18,\yy+7.5);
    
    \def\xx{19}
    \def\yy{-32}
    
    \draw[step=1, gray, very thin, xshift=\xx cm, yshift = \yy cm] (0,0) grid (5,5);
	\draw[step=5, black, very thin, xshift=\xx cm, yshift = \yy cm] (0,0) grid (5,5);
	\pgfmathsetmacro{\x}{\xx+1}
	\pgfmathsetmacro{\y}{\yy+5}
	\draw[step=1, gray, very thin,xshift=\x cm, yshift = \y cm] (0,0) grid (4,8);
	\draw[step=4, black, very thin,xshift=\x cm, yshift = \y cm] (0,0) grid (4,8);
	\pgfmathsetmacro{\x}{\xx+5}
	\pgfmathsetmacro{\y}{\yy+9}
	\draw[step=1, gray, very thin,xshift=\x cm, yshift = \y cm] (0,0) grid (8,4);
	\draw[step=4, black, very thin,xshift=\x cm, yshift = \y cm] (0,0) grid (8,4);
	\pgfmathsetmacro{\x}{\xx+9}
	\pgfmathsetmacro{\y}{\yy+5}
	\draw[step=1, gray, very thin,xshift=\x cm, yshift = \y cm] (0,0) grid (4,4);
	\draw[step=4, black, very thin,xshift=\x cm, yshift = \y cm] (0,0) grid (4,4);
	\pgfmathsetmacro{\x}{\xx+9}
	\pgfmathsetmacro{\y}{\yy+0}
	\draw[step=1, gray, very thin,xshift=\x cm, yshift = \y cm] (0,0) grid (5,5);
	\draw[step=5, black, very thin,xshift=\x cm, yshift = \y cm] (0,0) grid (5,5);
	
    \foreach \x in {1,...,4}{
		\draw (\xx+\x+8.5,\yy+4.5) node[black] {$0$};
	}

    \draw (\xx+13.5,\yy+0.5) node[RoyalBlue] {$1$};
    \draw (\xx+10.5,\yy+3.5) node[black] {$\star$};
    
    \draw[->]  (\xx+16,\yy+7.5) to (\xx+18,\yy+7.5);
    
    \def\xx{38}
    \def\yy{-32}
    
    \draw[step=1, gray, very thin, xshift=\xx cm, yshift = \yy cm] (0,0) grid (5,5);
	\draw[step=5, black, very thin, xshift=\xx cm, yshift = \yy cm] (0,0) grid (5,5);
	\pgfmathsetmacro{\x}{\xx+1}
	\pgfmathsetmacro{\y}{\yy+5}
	\draw[step=1, gray, very thin,xshift=\x cm, yshift = \y cm] (0,0) grid (4,8);
	\draw[step=4, black, very thin,xshift=\x cm, yshift = \y cm] (0,0) grid (4,8);
	\pgfmathsetmacro{\x}{\xx+5}
	\pgfmathsetmacro{\y}{\yy+9}
	\draw[step=1, gray, very thin,xshift=\x cm, yshift = \y cm] (0,0) grid (8,4);
	\draw[step=4, black, very thin,xshift=\x cm, yshift = \y cm] (0,0) grid (8,4);
	\pgfmathsetmacro{\x}{\xx+9}
	\pgfmathsetmacro{\y}{\yy+5}
	\draw[step=1, gray, very thin,xshift=\x cm, yshift = \y cm] (0,0) grid (4,4);
	\draw[step=4, black, very thin,xshift=\x cm, yshift = \y cm] (0,0) grid (4,4);
	\pgfmathsetmacro{\x}{\xx+9}
	\pgfmathsetmacro{\y}{\yy+0}
	\draw[step=1, gray, very thin,xshift=\x cm, yshift = \y cm] (0,0) grid (5,5);
	\draw[step=5, black, very thin,xshift=\x cm, yshift = \y cm] (0,0) grid (5,5);
	
    \foreach \x in {0,...,3}{
		\draw (\xx+\x+1.5,\yy+4.5) node[black] {$0$};
	}
    \draw (\xx+13.5,\yy+0.5) node[RoyalBlue] {$1$};
    \draw (\xx+0.5,\yy+0.5) node[black] {$\star$};
    
    \draw[->]  (\xx+16,\yy+7.5) to (\xx+18,\yy+7.5);
    
    \def\xx{0}
    \def\yy{-48}
    
    \draw[step=1, gray, very thin, xshift=\xx cm, yshift = \yy cm] (0,0) grid (5,5);
	\draw[step=5, black, very thin, xshift=\xx cm, yshift = \yy cm] (0,0) grid (5,5);
	\pgfmathsetmacro{\x}{\xx+1}
	\pgfmathsetmacro{\y}{\yy+5}
	\draw[step=1, gray, very thin,xshift=\x cm, yshift = \y cm] (0,0) grid (4,8);
	\draw[step=4, black, very thin,xshift=\x cm, yshift = \y cm] (0,0) grid (4,8);
	\pgfmathsetmacro{\x}{\xx+5}
	\pgfmathsetmacro{\y}{\yy+9}
	\draw[step=1, gray, very thin,xshift=\x cm, yshift = \y cm] (0,0) grid (8,4);
	\draw[step=4, black, very thin,xshift=\x cm, yshift = \y cm] (0,0) grid (8,4);
	\pgfmathsetmacro{\x}{\xx+9}
	\pgfmathsetmacro{\y}{\yy+5}
	\draw[step=1, gray, very thin,xshift=\x cm, yshift = \y cm] (0,0) grid (4,4);
	\draw[step=4, black, very thin,xshift=\x cm, yshift = \y cm] (0,0) grid (4,4);
	\pgfmathsetmacro{\x}{\xx+9}
	\pgfmathsetmacro{\y}{\yy+0}
	\draw[step=1, gray, very thin,xshift=\x cm, yshift = \y cm] (0,0) grid (5,5);
	\draw[step=5, black, very thin,xshift=\x cm, yshift = \y cm] (0,0) grid (5,5);
	
    \foreach \y in {0,...,3}{
		\draw (\xx+7.5,\yy+\y+9.5) node[black] {$0$};
	}
    \draw (\xx+13.5,\yy+0.5) node[RoyalBlue] {$1$};
    \draw (\xx+0.5,\yy+0.5) node[black] {$\star$};

\end{tikzpicture}

\caption{\label{fig:exchangepath}Permuting the marked particle $\color{RoyalBlue}1$ with a site $\star$.}

\end{figure}

\begin{defn}
Let $M$ be a $T$-step move and $\eta$ such that $\eta\left(0\right)=1$.
At each time $t$ we can track the position of the particle started
at $0$, denoting it by $X_{t}$. Then the translated move $M^{\tau}$
is given by $M_{t}^{\tau}\eta=\tau_{X_{t}}M_{t}\eta$.
 with $x_{t}^{\tau}=x_{t}-X_{t}$ and $y_{t}^{\tau}=y_{t}-X_{t}$, so that for all $t$, $M^\tau_t\eta(0)=1$. 
\end{defn}

\begin{proof}
[Proof of inequality (\ref{comparaison})] We now use the move defined above in order to compare
both diffusion matrices. First, note that for $\eta\in\text{Dom}_0M$
and $i \in \zz_\l^d$ a neighbor of the origin,
\[
e_1\cdot i+f\left(\tau_{i}\eta^{0,i}\right)-f\left(\eta\right)=\sum_{t=1}^{T}\One_{x_{t}^{\tau}=0}e_1\cdot y_{t}^{\tau}+f\left(M_{t}^{\tau}\eta\right)-f\left(M_{t-1}^{\tau}\eta\right).
\]
By the Cauchy-Schwarz inequality
\[
\left(e_1\cdot i+f\left(\tau_{i}\eta^{0,i}\right)-f\left(\eta\right)\right)^{2}\le T\sum_{t=1}^{T}c_{x_{t}^{\tau}y_{t}^{\tau}}(M_{t}^{\tau}\eta)\,\left(\One_{x_{t}^{\tau}=0}e_1\cdot y_{t}^{\tau}+f\left(M_{t}^{\tau}\eta\right)-f\left(M_{t-1}^{\tau}\eta\right)\right)^{2},
\]
where we have used the fact that, by definition of a move, $c_{x_t y_t}(M_t \eta)=1$.
Therefore,
\begin{multline*}
\mu_{0}\left[\overline{\eta}_{0i}\left(e_1\cdot i+f\left(\tau_{i}\eta^{0,i}\right)-f\left(\eta\right)\right)^{2}\right]\\
\le T\mu_{0}\left[\overline{\eta}_{0i}\sum_{t=1}^{T}c_{x_{t}^{\tau}y_{t}^{\tau}}(M_{t}^{\tau}\eta)\,\left(\One_{x_{t}^{\tau}=0}e_1\cdot y_{t}^{\tau}+f\left(M_{t}^{\tau}\eta\right)-f\left(M_{t-1}^{\tau}\eta\right)\right)^{2}\right]\\
=T\sum_{\eta}\mu_{0}\left(\eta\right)\overline{\eta}_{0i}\sum_{t}\sum_{\eta^{\prime}}\sum_{y\sim0}\One_{\eta^{\prime}=M_{t-1}^{\tau}\eta}\One_{0=x_{t}^{\tau}}\One_{y=y_{t}^{\tau}}c_{0y}\left(\eta^{\prime}\right)\left(u\cdot y+f\left(\tau_{y}\eta^{\prime0y}\right)-f\left(\eta^{\prime}\right)\right)^{2}\\
+T\sum_{\eta}\mu_{0}\left(\eta\right)\overline{\eta}_{0i}\sum_{t}\sum_{\eta^{\prime}}\sum_{x\ne0}\sum_{y\sim x}\One_{\eta^{\prime}=M_{t-1}^{\tau}\eta}\One_{x=x_{t}^{\tau}}\One_{y=y_{t}^{\tau}}c_{xy}\left(\eta^{\prime}\right)\left(f\left(\eta^{\prime xy}\right)-f\left(\eta^{\prime}\right)\right)^{2}\\
\le T^{2}2^{\text{Loss}M}\sum_{\eta^{\prime}}\sum_{y\sim0}\mu_{0}\left(\eta^{\prime}\right)c_{0y}\left(\eta^{\prime}\right)\left(e_1\cdot y+f\left(\tau_{y}\eta^{\prime0y}\right)-f\left(\eta^{\prime}\right)\right)^{2}\\
+T^{2}2^{\text{Loss}M}\sum_{\eta^{\prime}}\sum_{x\neq0}\sum_{y\sim x}\mu_{0}\left(\eta^{\prime}\right)c_{xy}\left(\eta^{\prime}\right)\left(f\left(\eta^{\prime xy}\right)-f\left(\eta^{\prime}\right)\right)^{2}.
\end{multline*}
In order to obtain the last inequality, we first replaced $\mu_0(\eta)$ with $\mu_0(\eta')$. Then, given $\eta',t,x_t^\tau$ and $y_t^\tau$, we estimated the sum $\sum_{\eta}\One_{\eta'=M_{t-1}^\tau}\One_{x=x_t^\tau}\One_{y=y_t^\tau} \le 2^{\text{Loss}(M)}$. Next, the sum over $t$ contributes in another factor $T$, and finally, we bound $\overline{\eta}_{0i} \le 1$.\\

We have hence shown that:  $e_1\cdot D_{\text{aux}}e_1 \le T^{2}2^{\text{Loss}M} D(q)$.\\

Now, plugging the appropriate bounds for $\text{Loss}M$, $L$ and $\l$ gives:

\begin{itemize}
    \item If $k=2$, $$ T^{2}2^{\text{Loss}M} \le \exp(c \log(\l) \l) \le \exp\left(c \log(\nicefrac{1}{q})^{d} q^{-\frac{1}{d-1}}\right).$$
    
    \item If $k\ge3$, $$T^{2}2^{\text{Loss}M} \le \exp(c \l^d) \le \exp\left(c \, \text{exp}_{(k-2)}(q^{-\frac{1}{d-k+1}})^d\right) \le \text{exp}_{(k-1)}(c q^{-\frac{1}{d-k+1}}).$$
\end{itemize}

This concludes the proof of the inequality (\ref{comparaison}). Together with Proposition \ref{Daux}, we obtain the lower bound of Theorem \ref{mainthm}.
\end{proof}

\section{Upper bound}

In order to find an upper bound, we will look for a suitable test function
to plug in \eqref{variational_principle}. Without loss of generality
we consider $u=e_{1}$.

Let 
\begin{equation}
\label{eq:l_upperbound}
\l = \begin{cases}
cq^{-\frac{1}{d-1}} & \text{if} \quad k=2, \\
\exp_{(k-2)}\left(cq^{-\frac{1}{d-k+1}}\right) & \text{if} \quad k\ge3,
\end{cases}
\end{equation}
for $c>0$ that may depend on $d$ and $k$ but not on $q$. 

We now define the $k$ \emph{neighbor bootstrap percolation} on the box $\left[-\l,\l\right]^{d}$.
It is usually seen as a deterministic process defined on $\Omega_\l := \left\{ 0,1\right\} ^{\left[-\l,\l\right]^{d}}$, but for our needs it is convenient to define the \emph{bootstrap percolation map} $\BP : \Omega_\l \rightarrow \Omega_\l$:
\[
\BP(\eta)(x)=\begin{cases}
0 & \text{if }\eta(x)=0,\\
0 & \text{if }\sum_{y\sim x \in \left[-\l,\l\right]^{d}}\left(1-\eta(y)\right)\ge k,\\
1 & \text{otherwise}.
\end{cases}
\]
That is, empty sites remain empty and occupied sites become empty if they have at least $k$ empty neighbors. Let $\BP^\infty (\eta)$ be the limiting configuration, that is:
\[
\BP^\infty(\eta)(x)=\begin{cases}
0 & \text{if } \exists t \in \nn, \BP^t(\eta)(x)=0,\\
1 & \text{otherwise}.
\end{cases}
\]
For more details about bootstrap percolation, see e.g. \cite{M17}. 

The following observation clarifies the relation between bootstrap percolation and the KA model.
\begin{observation}
\label{obs:BPandKA}Fix $\eta \in \Omega_\l$. Let $x,y$ two sites in $[-\l,\l]^d$ such that $c_{xy}(\eta) = 1$. Then $\BP^\infty(\eta) = \BP^\infty(\eta^{xy})$.
\end{observation}

\begin{proof}
Since in a legal KA move the particle has at least $k$ empty
neighbors both before and after the exchange, it would be emptied
for both states.
\end{proof}
We say that a site $x$ is in the bootstrap percolation cluster of the origin if there
is a nearest neighbor path $0,x_{1},\dots,x_{n}=x$ such that $\BP^{\infty}(\eta)\left(x_{i}\right)=0$
for $i=1,\dots,n$. For any $\eta \in \Omega$, let $f(\eta)$ be the first coordinate of the rightmost
site in the bootstrap percolation cluster of the origin for $\restriction{\eta}{[-\l,\l]^d}$.

\begin{defn}
Fix $\eta \in \Omega_\l$. We define $\mathcal{B}$ as the event that the bootstrap percolation cluster of the origin contains a site of $\infty$-norm at least $\l-1$.  
\end{defn}

\begin{claim}\label{bootstrapestimate}
\label{claim:bpcluster_small}There exists a constant $\gamma>0$
such that $\mu\left(\mathcal{B}\right)\le e^{-\gamma \l}$.
\end{claim}

\begin{proof}
This is a direct consequence of \cite[Lemma 5.1]{CM} and the
choice of $\l$ in (\ref{eq:l_upperbound}), as we set $c = \beta_{-}(d,k)$ following the notations of \cite{CM}.
\end{proof}

\begin{defn}
Fix $x,y\in\zz^{d}$ and $\eta \in \Omega$. We say that the edge $xy$ is \emph{pivotal} for $\eta$
if $c_{xy}(\eta)=1$ and $f\left(\eta\right)\neq f\left(\eta^{xy}\right)$.
\end{defn}

\begin{claim}
Fix $x,y \in \zz^d$. The edge $xy$ can only be pivotal if at least one of its endpoints is on the inner boundary of $\left[-\l,\l\right]^{d}$, and $\mathcal{B}$ must occur for either $\eta$ or $\eta^{xy}$.
\end{claim}

\begin{proof}
If both $x$ and $y$ are outside $[-\l,\l]^d$, $xy$ is not pivotal.
If both are inside $[-\l+1,\l-1]^d$, then $c_{xy}(\eta) = c_{xy}(\restriction{\eta}{[-\l,\l]^d})$. By Observation \ref{obs:BPandKA}, $xy$ can not be pivotal. 

Assume $xy$ is pivotal. Then either $x$ or $y$ is in the bootstrap percolation cluster of the origin for either $\eta$ or $\eta^{xy}$, hence $\mathcal{B}$ must occur for either $\eta$ or $\eta^{xy}$.
\end{proof}
We can now estimate the right hand side of \eqref{variational_principle} for
our choice of $f$.

\begin{align*}
\sum_{x\neq0}\sum_{y\sim x}\mu_{0}\left[c_{xy}\left(f\left(\eta^{xy}\right)-f\left(\eta\right)\right)^{2}\right] & \le\sum_{x\neq0}\sum_{y\sim x}\mu_{0}\left[c_{xy}\,4\l^{2}\,\One_{f\left(\eta^{xy}\right)\neq f\left(\eta\right)}\right]\\
 & \le4\l^{2} \cdot 2d \ \#\{x \ | \ ||x||_{\infty} = \l \} \cdot  \sup_{x,y}\mu_{0}\left[\One_{x,y\text{ pivotal}}\right]\\
 & \le c l^{d+2} e^{-\gamma \l}.
\end{align*}

For the second term, fix $y \sim 0$, and assume first that $\eta, \eta^{0y} \notin \mathcal{B}$, and $c_{0y}=1$. In this case, the bootstrap percolation cluster of the origin
for $\tau_{y}\eta^{0y}$ is a translation by $-y$ of the bootstrap percolation cluster
for $\eta$. Therefore, the term $e_{1}\cdot y+f\left(\tau_{y}\eta^{0y}\right)-f\left(\eta\right)$
equals $0$. In the other case we bound $\left|f\left(\tau_{y}\eta^{0y}\right)-f\left(\eta\right)\right|\le2\l$
and $\left|u\cdot y\right|\le1$, obtaining 

\begin{align*}
\sum_{y\sim0}\mu_{0}\left[c_{0y}\left(u\cdot y+f\left(\tau_{y}\eta^{0y}\right)-f\left(\eta\right)\right)^{2}\right] & \le4d\,\left(2\l + 1\right)\mu_0(\mathcal{B})\\
 & \le c\l e^{-\gamma \l}.
\end{align*}
Summing both contributions now yields the expected bounds.\qed

\section{Further questions}
Theorem \ref{mainthm} shows how the self-diffusion constant decays as $q \to 0$ up to a constant for $k \ge 3$ and logarithmic correction for $k=2$. In \cite[equation (6.26)]{CphD}, it is conjectured for the case $k=d=2$ that the true behavior is $D(q) \approx \exp(-\gamma q^{-1})$ for $\gamma = \frac{\pi^2}{9}$. In view of recent works related to the Fredrickson-Andersen model \cite{HMT}, where similar scaling is observed and the exact constant $\gamma$ could be identified, it seems reasonable that such a result could also be obtained for the Kob-Andersen model.

The methods used here could also be applied in other models for which the combinatorial structure allows a construction of a $T$-step move as in Lemma \ref{existenceTstepmove}. In particular, it is natural to consider other kinetically constrained lattice gases, or even look for universality results on the self-diffusion coefficient.  

\section*{Acknowledgements}
We would like to thank Oriane Blondel and Cristina Toninelli for our useful discussions. 
A.S. acknowledges the support of the ERC Starting Grant 680275 MALIG. This project has been partially supported by the ANR grant LSD (ANR-15-CE40-0020). 

\bibliographystyle{plain}
\bibliography{bibliography}

\vskip 2cm  

\end{document}